\documentclass[a4paper, twoside]{article} 

\usepackage[
	top = 1.15 in, 
	bottom = 1.25 in,
	left = 1.15 in, 
	right = 1.15 in,
	includehead]{geometry}

\usepackage{scrextend}
\changefontsizes{11 pt}

\usepackage[all]{nowidow}
\usepackage[utf8x]{inputenc}
\usepackage{microtype}
\usepackage{mathptmx}
\usepackage[T1]{fontenc}

\usepackage{booktabs}

\newcommand{\myauthor}[3]{
\noindent
\begin{minipage}[t]{.45\textwidth}
	\begin{flushright}
		\textsc{#1} \\
		{\footnotesize\texttt{#2}}
	\end{flushright} 
\end{minipage}
\qquad
\begin{minipage}[t]{.45\textwidth}
#3
\end{minipage}
}

\usepackage[utf8x]{inputenc}
\usepackage{amsfonts,amscd,amssymb,amsmath,amsrefs}
\usepackage{graphicx}
\usepackage[british]{babel}
\usepackage{caption}
\usepackage{mathdots}
\usepackage{mathtools} 
\usepackage{subfigure}
\usepackage{cancel}

\usepackage{makeidx}
\usepackage{multicol}
\usepackage{array}
\usepackage[pdf,all]{xy}
\xyoption{line}
\CompileMatrices

\usepackage{sidecap}

\mathchardef\hy="2D % Define a "math hyphen"

\usepackage{pgf,tikz}
\usepackage{tikz-cd}
\usetikzlibrary{calc}
\usetikzlibrary{matrix,arrows}
\usepackage{stackrel}
\usepackage[shortlabels]{enumitem} %Para listas estilo Weibel
\usepackage{stmaryrd} %Para double brackets
\usepackage{etoolbox}
\usetikzlibrary{trees}

\patchcmd{\section}{\normalfont}{\normalfont\large}{}{}

\usepackage[bb=ams, cal=euler, scr=rsfso , frak=euler]{mathalpha}

\usepackage{fancyhdr}
\fancypagestyle{references}
	{\fancyhead[RE]{\small\it References}
	\fancyhead[LO]{\small\it References}
	\fancyhead[RO,LE]{\small\bf\thepage}
	\fancyfoot[L,R,C]{}
	}

\renewcommand{\AA}{\mathcal{A}}

\DeclareMathOperator{\cone}{cone}

\newcommand{\Ad}{\operatorname{Ad}}

\usepackage{enumitem}
\setlist[enumerate]{label= (\arabic*)}

\DeclareMathOperator\Def{Def}

\newenvironment{tenumerate}{
 \begin{enumerate}
  \setlength{\itemsep}{0pt}
  \setlength{\parskip}{0pt}
}{\end{enumerate}}

\newenvironment{titemize}{
\begin{itemize}
  \setlength{\itemsep}{0pt}
  \setlength{\parskip}{0pt}
}{\end{itemize}}

\definecolor{trinityblue}{rgb}{0.05, 0.45,0.75}
\definecolor{newcol}{rgb}{0,0,0}
\definecolor{deepblue}{rgb}{0.05, 0.45,0.75}
\DeclareTextFontCommand{\new}{\color{black}\em}
\usepackage[pdftex, colorlinks,bookmarks 
= true,bookmarksnumbered = true]{hyperref}

\hypersetup{colorlinks,linkcolor={deepblue},
	citecolor={deepblue},urlcolor={deepblue}}  
\usepackage{sectsty}
\chapterfont{\color{newcol}}  % sets colour of chapters
\sectionfont{\color{newcol}}  % sets colour of sections
\subsectionfont{\color{newcol}}  % sets colour of sections

\makeatletter
\AtBeginDocument{%
  \let\[\@undefined
  
\DeclareRobustCommand{\[}{\begin{equation}}%
  \let\]\@undefined
  
\DeclareRobustCommand{\]}{\end{equation}}%
}
\makeatother 
\mathtoolsset{showonlyrefs,showmanualtags}

\usepackage{amsthm}
\usepackage{thmtools}
\newtheoremstyle{mytheorem}
  {\topsep}   % ABOVESPACE
  {\topsep}   % BELOWSPACE
  {\itshape}  % BODYFONT
  {0pt}       % INDENT (empty value is the same as 0pt)
  {\bfseries\color{newcol}} % HEADFONT
  {\color{newcol}}         % HEADPUNCT
  {5pt plus 1pt minus 1pt} % HEADSPACE
  {}          % CUSTOM-HEAD-SPEC
  
\theoremstyle{mytheorem}
\newtheorem{theorem}{Theorem}[section]
\newtheorem{prop}[theorem]{Proposition}
\newtheorem{corollary}[theorem]{Corollary}
\newtheorem{lemma}[theorem]{Lemma}
\newtheorem*{conj*}{Conjecture}

 \newtheoremstyle{introthm}
  {\topsep}   % ABOVESPACE
  {\topsep}   % BELOWSPACE
  {\itshape}  % BODYFONT
  {0pt}       % INDENT (empty value is the same as 0pt)
  {\bfseries\color{newcol}} % HEADFONT
  {\color{newcol}{.}}         % HEADPUNCT
  {5pt plus 1pt minus 1pt} % HEADSPACE
  {}          % CUSTOM-HEAD-SPEC

\theoremstyle{introthm}
\newtheorem{introthm}{Theorem}

\newcommand{\ind}{\operatorname{Ind}_\PP}
\newcommand{\cof}{\rightarrowtail}

\newcommand{\weq}{\overset{\sim}{\longrightarrow}}

\newtheoremstyle{mydefinition}
  {\topsep}   % ABOVESPACE
  {\topsep}   % BELOWSPACE
  {}  % BODYFONT
  {0pt}       % INDENT (empty value is the same as 0pt)
  {\bfseries\color{newcol}} % HEADFONT
  {\color{newcol}}         % HEADPUNCT
  {5pt plus 1pt minus 1pt} % HEADSPACE
  {}          % CUSTOM-HEAD-SPEC
  
\theoremstyle{mydefinition}
\newtheorem{definition}[theorem]{Definition}

\newtheorem{rmk}[theorem]{Remark}

\newtheoremstyle{mydefinition2}
  {\topsep}   % ABOVESPACE
  {\topsep}   % BELOWSPACE
  {}  % BODYFONT
  {0pt}       % INDENT (empty value is the same as 0pt)
  {\bfseries\color{newcol}} % HEADFONT
  {\color{newcol}{.}}         % HEADPUNCT
  {5pt plus 1pt minus 1pt} % HEADSPACE
  {}          % CUSTOM-HEAD-SPEC
  
\theoremstyle{mydefinition2}
\newtheorem*{definition*}{Definition}
\newtheorem*{remark*}{Remark}
\newtheorem*{obs*}{Observation}
\newtheorem*{example*}{Example}

\theoremstyle{plain} % just in case the style had changed
\newcommand{\thistheoremname}{}
\newtheorem{genericthm}[theorem]{\thistheoremname}

\newcommand{\imor}{\interleave\kern-.45em\longrightarrow}

\newcommand{\Der}{\operatorname{Der}}
\newcommand{\HH}{\mathrm{HH}}
\newcommand{\AQ}{\mathscr{H}}
\renewcommand{\H}{\mathrm{H}}
\newcommand{\FF}{\mathcal F}

\definecolor{sqsqsq}{rgb}{0.13,0.13,0.13}
\definecolor{aqaqaq}{rgb}{0.63,0.63,0.63}

\newcommand\id{\mathrm{id}}

\newcommand{\SMod}{{}_\Sigma\mathsf{dgMod}}

\renewcommand{\tt}{\otimes}

\newcommand\cls[1]{\llbracket#1\rrbracket}
\newcommand{\NN}{\mathbb N}

\newcommand{\kk}{\Bbbk}
\newcommand{\Aut}{\operatorname{Aut}}

\newcommand{\Tor}{\operatorname{Tor}}
\newcommand{\End}{\operatorname{End}}

\newcommand{\Cog}{\mathsf{Cog}}
\newcommand{\Alg}{\mathsf{Alg}}
\newcommand{\Cxs}{\mathsf{Ch}_\kk}

\newcommand{\Cell}{\operatorname{Cell}}
\newcommand{\Sing}{\operatorname{Sing}}
\newcommand{\Sull}{A_{\mathrm{PL}}}

\newcommand{\Q}{{\mathbb{Q}}}

\newcommand{\PP}{{\mathcal{P}}}

\definecolor{newterm-color}{RGB}{0, 0, 0}

\theoremstyle{mytheorem}
\newtheorem*{theorem*}{Theorem}

\theoremstyle{plain} % just in case the style had changed
\renewenvironment{abstract}{%
\small\begin{center}
\begin{minipage}{.9\textwidth}
\small
}
{\par\noindent\end{minipage}\end{center}\vspace{3 em}}
\makeatletter
\renewcommand\@maketitle{%
\hfill
\begin{center}\begin{minipage}{0.9 	\textwidth}
\centering
\vskip 6em
\let\footnote\thanks 
{\LARGE \@title \par }
\vspace{1 em}
\vskip 1 em
{\large \@author \par}
\vspace{3.5 em}

\end{minipage}\end{center}
\par
}
\makeatother

\tikzcdset{arrow style=tikz, diagrams={>=stealth}}

\usepackage{titletoc}

\titlecontents{section}% <section>
[0.2em]% <left>
{\small}% <above-code>
{\thecontentslabel.\hspace{3pt}}%<numbered-entry-format>; you could also 
{}% <numberless-entry-format>
{\enspace\titlerule*[0.5pc]{.}\contentspage}%<filler-page-format>
\titlecontents*{subsection}% <section>
[1em]% <left>
{\footnotesize}% <above-code>
{\thecontentslabel. \hspace{3pt}}% <numbered-entry-format>; you could also 
{}% <numberless-entry-format>
{}% <filler-page-format>
[ --- \ ]% <separator>
[]% <end>

\title{\vspace{- 1.25 in}{\textbf{A spectral sequence for tangent cohomology of algebras over algebraic operads}}}
\author{\textsc{Jos\'e M. Moreno-Fern\'andez \\
	Pedro Tamaroff}}
\date{}
\begin{document}
\maketitle

\begin{abstract}
\small{
Operadic tangent cohomology generalizes the existing 
cohomology theories of Chevalley--Eilenberg, Hochschild,  
and Harrison to address
the deformation theory of general types of algebras
through gadgets known as deformation
complexes. The cohomology of these is in general
very non-trivial to compute, and in this paper we complement
the existing computational techniques by producing 
a spectral sequence that converges to 
the operadic cohomology of a fixed algebra.
Our main technical tool is that of filtrations
arising
from towers of cofibrations of algebras,
which play the same role cell attaching maps and
skeletal filtrations do for topological spaces. 

As an application, we consider the rational Adams--Hilton construction
on topological spaces, where our spectral sequence gives 
rise to a seemingly new and completely algebraic description 
of the Serre spectral sequence, which we also show is multiplicative
and converges to the Chas--Sullivan loop product. We also
consider relative Sullivan--de Rham models of a fibration $p$, where our 
spectral sequence converges to the rational homotopy groups of the identity component of the space of self-fiber-homotopy equivalences
of $p$.}

\bigskip


\medskip


\end{abstract}
\thispagestyle{empty}
\vspace{-2 em}
 \section{Introduction}\label{sec:intro}

\definecolor{bostonred}{rgb}{0.8,0,0}

Algebraic operads are an effective gadget to study
different types of algebras through a common
language. In particular, they provide us with tools
to study the deformation theory of such algebras.
For example, the operads controlling Lie, associative
and commutative algebras ---collectively known as the
three graces of J.-L. Loday--- along with the operadic
formalism, recover for us swiftly the usual (co)homology
theories of Chevalley--Eilenberg~\cite{Chevalley1948}, Hochschild~\cite{Hochschild} and Harrison~\cite{Harrison},
shedding light on the relation between these three. We point the reader to~\cite{Griffin} for an interesting
example of this.

More generally, 
for an operad $\PP$ and a $\PP$-algebra $A$, there is
a dg Lie algebra $\Def_\PP(\id_A)$, the 
deformation complex of $A$ (also known as the tangent Lie algebra of $A$), that codifies all deformation
problems over $A$, in the spirit
of~\cite{dotsenko2018twisting,DoubekMarklDeformation,FoxDeformation,calaque2019moduli}. In particular,
given a deformation problem, classes in the
cohomology groups $\AQ^*(A)$ of $\Def_\PP(\id_A)$ ---which are now known
as the Andr\'e--Quillen or tangent cohomology
groups of $A$--- allow us, among
other things, to determine obstructions to the existence of solutions
of such deformation problems. In this paper, we construct
a spectral sequence whose input is, in a precise sense,
a ``cellular decomposition of $A$'', that converges
to the cohomology groups $\AQ^*(A)$. 

There is a rich interplay~\cite{Hinich1997,vallette2014homotopy} 
between the homological algebra 
that arises when studying such deformation complexes
and the homotopical algebra of D. Quillen~\cite{Quillen} 
and, in particular, of it with the study of the 
homotopy category of some type of algebras.

Our main interest lies in 
towers of cofibrations of $\PP$-algebras for a
fixed operad $\PP$. 
For simplicity, we assume
$\PP$ is non-dg, although most of what we do can
be extended without too much effort for dg operads.
We focus on the computation of the tangent
cohomology of $\PP$-algebras once such a tower of cofibrations has been fixed. 

The relation between tangent cohomology and
(co)derivations of types of algebras has appeared many 
times in the literature, starting with the pioneering
work of M. Schlessinger and J. Stasheff in~\cite{Intrinsic,SchStash}. 
A definition for algebras over operads then naturally followed; 
we point the reader to~\cite{Milles} for an account of this development.

\textit{A remark on terminology.}
The term `Andr\'e--Quillen cohomology' seems to have been first used
in this generality by Goerss--Hopkins~\cite{Goerss2000}, while 
V. Hinich~\cite{Hinich1997} prefers the term `(absolute) cohomology'. 
In order to avoid any confusion with the classical Andr\'e--Quillen
cohomology for commutative algebras over fields of positive
characteristic, whose definition also happens to coincide with that of 
Quillen (co)homology, we will use the more neutral term
`tangent cohomology', as in~\cite{SchStash}. 

\begin{definition*}
Let $f:B_0\longrightarrow B$ be a morphism of $\PP$-algebras, 
and $M$ a $B$-module. The \new{tangent cohomology of $B$ relative
to $B_0$ with values in $M$} is the cohomology of the cochain complex
$\Def_\PP(f,M) := \Der_{Q_0}(Q,M)$ where $Q_0\longrightarrow Q$
is a cofibrant replacement of $f:B_0\longrightarrow B$, and we
write it $\AQ^*_{B_0}(B,M)$.
\end{definition*}

As we will explain, there is a functor from triples
of algebras to left-exact sequences ---given by the
usual Jacobi--Zariski sequence in geometry~\cite{Iyengar}*{Section 2.4}--- which yields a short exact sequence as 
long as the second map is a cofibration.
This interplay between derivations and cofibrations 
is at the very heart of the construction of the
spectral sequence appearing in
our main statement,~Theorem~\ref{thm:SucesionEspectral}.
There, we show that if
one can exhibit $Q_0\longrightarrow Q$ as a colimit of
a tower cofibrations
$
\mathcal T: Q_0 
	\cof Q_1  
		\cof \cdots 
			\cof Q_s 
				\cof Q_{s+1} 
					\cof \cdots 
						\cof \varinjlim_s Q_s = Q,
$
then one obtains a corresponding spectral sequence in terms of the tangent
cohomology of the stages of $\mathcal T$, where we allow coefficients to
take values in a $B$-module, and not necessarily $B$ itself. 

%To guarantee
%convergence of this spectral sequence, we require a natural technical condition
%on the tower, which already appears in~\cite{Cirici2019}*{Definition 3.2}. Namely,
%we require that each map $Q_s \cof Q_{s+1}$ is a \emph{relative} Koszul--Sullivan extension, in the sense that the new generators added to $Q_{s+1}$ have boundaries
%in~$Q_s$. Following~\cite{Cirici2019}, we call such towers \emph{Koszul--Sullivan towers}, and remark that
%we are working relative to a starting $\PP$-algebra $Q_0$ as opposed to what is done
%there. We also remark that in \emph{loc. cit.} the authors show that
%such towers exist for a large class of operads $\PP$ and their $\PP$-algebras,
%so that one can indeed produce towers to input into our spectral sequence. With this
%at hand, our main result is the following:

\begin{introthm} Let $Q$ be the colimit of a tower $\mathcal T$ of 
cofibrations of $\;\PP$-algebras.  
There is a functorial right half-plane spectral sequence with first page 
\[ 
	E^{s,t}_1 = H^{s+t} 	(\Der_{Q_s}(Q_{s+1},-))	\;	
			\xRightarrow{\phantom{m}s\phantom{m}} \; 
							H^{s+t}(\Der_{Q_0}(Q,-)).\] 
\end{introthm}

The broad generality in which this spectral sequence appears implies
that further conditions are necessary for it to converge,
with the target being the one mentioned above.
Since $Q_0\longrightarrow Q$ is a cofibration of $\;\PP$-algebras that
is a cofibrant replacement of the morphism of $\;\PP$-algebras $B_0\longrightarrow B$,
the target of this spectral sequence is $\AQ^*_{B_0}(B,-)$. This immediate reinterpretation of the theorem is recorded in Corollary~\ref{cor:AQss};  
this is the situation we are interested in general.

Having done this, we study how our construction specializes when considering the 
algebraic models of J. F. Adams and P. J. Hilton~\cite{AH}, and 
D.~Sullivan~\cite{Infinitesimal} in rational homotopy theory. 
One of the main features of these algebraic models
of rational homotopy types is that they are built through towers of cofibrations in terms
of a cellular decomposition of choice. Our first result in this direction is 
Theorem~\ref{thm:Adams-Hilton}, which the reader is invited to compare
with that of S. Shamir~\cite{Shamir} and to the eponymous spectral sequence of 
J.-P. Serre. We expect this spectral sequence, in which all coefficients are 
\emph{rational}, to be isomorphic to the one obtained in~\cite{Cohenloops} by 
R. L. Cohen, J. D. S. Jones and J. Yan. It gives us a completely algebraic description of 
a multiplicative spectral sequence converging to the Chas--Sullivan loop product in 
string  topology for simply connected oriented closed manifolds.

\begin{introthm}\label{Intro:AH}
Let $X$ be a CW complex with exactly one $0$-cell, no $1$-cells and all its 
attaching maps are based with respect to the only $0$-cell. There is a convergent
first quadrant spectral sequence with
\[ 
	E_2^{s,t}  = \hom(H_s(X), H_t(\Omega X)) 
\;			\xRightarrow{\phantom{m}s\phantom{m}} \;
						H_{s+t}(LX),
			\]
Moreover, this is a 
spectral sequence of algebras, its product in the $E_2$-page is given by
the convolution product in $\hom(H_*(X), H_*(\Omega X))$. Whenever $X$ is
a simply connected oriented closed manifold, this product converges to 
the Chas--Sullivan loop product. 
\end{introthm}

We now turn our attention to Sullivan models. Let $p:E\longrightarrow B$ be a fibration
between simply-connected CW-complexes, with $E$ having finitely many cells. The grouplike 
topological monoid $\Aut(p)$ of self-fibre-homotopy equivalences of $p$ is an important 
object in algebraic topology; we point out as an example the classification theorem of 
J.~Stasheff~\cite{StasheffFibrations}, \cite{MayFibrations}*{Chapter 9}, and the more 
recent~\cite{Blomgren}. In~\cite{Fel10}, the authors show that the rational homotopy
type of the connected component $\Aut_1(p)$ of the identity of $\Aut(p)$ is determined 
completely in terms of the Harrison cohomology of the Sullivan model of $p$. We also 
point the reader to~\cite{BlockLazarev}, where a general relation between the
Harrison cohomology of Sullivan--de Rham algebras and homotopy types of function spaces 
is given, and to~\cite{BuijsMurillo}, where the authors show the rational homotopy 
groups of function spaces are also determined completely in terms of the Harrison 
cohomology of a Sullivan model. Building on top of the main result of \cite{Fel10}, 
we obtain the following; see Theorem~\ref{thm:SullivanFibrations}.

\begin{introthm}
	Let $F\hookrightarrow E \xrightarrow{p} B$ be 
	a fibration of $1$-connected CW-complexes, 
	with $E$ finite. There is a convergent
	spectral sequence with 
	\begin{equation*}
	E_2^{s,t} =  \hom(\pi_s(F),H^{t}(E)) \; \xRightarrow{\phantom{m}s\phantom{m}}
	\; \pi_{s-t}(\operatorname{Aut}_1(p)).
	\end{equation*}
\end{introthm}

Again, all coefficients above are rational. In this case, it should be possible to show
that the spectral sequence carries a multiplicative structure inherited from the 
convolution product on the $E_2$-page making it a spectral sequence of Lie algebras by
obtaining a result in the lines of Theorem 3 in~\cite{BuijsMurillo}, by replacing the
space of maps $\mathcal{F}(X,Y)$ by $\Aut_1(p)$. 
 
We remark that A. Berglund and B. Saleh~\cite{BerglundSaleh}*{Proposition 4.4}
have independently considered a spectral sequence of the same shape as 
ours in order to determine a Quillen (that is, dg Lie) model of the 
classifying space of  the grouplike monoid of homotopy automorphisms of a space 
that fix a given subspace. The same spectral sequence appears in the work of 
A. Berglund and I. Madsen, see~\cite{BerglundMadsenActa}*{Lemma 3.5}.
This gives a third direction in which our methods can be pushed towards,
and the interested reader can consult that paper for further details. 

The paper is organized as follows. In Section~\ref{sec:preliminares}
we briefly recall the elements of operad theory we will use, mainly to settle the notation.
In Section~\ref{sec:2}, we develop the main technical tool that we 
introduce in this paper: the spectral sequence of 
Theorem~\ref{thm:SucesionEspectral}. We conclude this section 
 by showing that our spectral sequence degenerates in some 
natural situations.
Finally, we give in Section~\ref{sec:AplicacionesRacionales}
the mentioned applications to rational homotopy theory.

\textbf{Conventions.}
Throughout, we write $\PP$ for a symmetric operad in the category of vector spaces
over a characteristic zero field $\kk$. We assume such an operad is reduced ($\PP(0) = 0$), 
and identify it with the corresponding Schur endofunctor $V\longmapsto \PP(V)$. All 
dg $\PP$-algebras are homologically $\mathbb{Z}$-graded unless stated otherwise.  
We also make use of the following functors and constructions: $\Sing_*$ is the
singular chains functor, $\Sing^*$ is the singular cochains functor, $\mathcal A$ 
is the Adams--Hilton construction, $\Cell$ is the cellular chain complex of a
CW-complex, $\Omega$ is the based loop space, and $L$ is the free loop space. 
The dual of a graded vector space $V$ is $V^\vee$. For a simply-connected 
topological space $X$, we write $\pi_*(X)^\vee = (\pi_*(X)\otimes \Q)^\vee$, omitting
the rational coefficients. All tensor products and $\hom$-sets are understood to 
be taken over the fixed characteristic zero base field $\kk$, which in Section 3
will be~$\Q$. Similarly, in that section, all (co)homology groups are assumed 
to have rational coefficients.

\subsection*{Acknowledgements} 

The first author thanks Aniceto Murillo for useful conversations, and acknowledges 
support from the IRC Postdoctoral Fellowship \texttt{GOIPD/2019/823},
from the Ministerio de Ciencia e Innovación (Spain) grant \texttt{PID2023-149804NB-I00},
and from the Junta de Andalucía research group \textrm{grupoPaidi-G-FEDER-FQM264}.
The second author thanks Andrea Solotar, Mariano Su\'arez-\'Alvarez and the 
participants of the homological algebra seminar in the University of Buenos
Aires for useful comments. We kindly thank Jim Stasheff for  useful comments 
and suggestions, and the anonymous referees for their careful reading of the
manuscript and for providing us with valuable feedback.

%%Section 1

\section{Operadic background}
\label{sec:preliminares}

This section starts with a recollection of the necessary background on operad theory, 
mainly to set up notation. An excellent reference is \cite{LodVal}.

\subsection{$\Sigma$-modules and operads}

{{}} Write $\SMod$ for the category of dg $\Sigma$-modules. Recall there is a monoidal 
product, the \new{composition product} 
	\[ -\circ-:\SMod\times \SMod
			\longrightarrow
				\SMod 
				\] 
with unit $1$ the module concentrated in arity $1$ where its value is $\kk$. 
It is useful to think of an element in $X\circ Y$ as a corolla whose only 
vertex is labelled by $x\in X$ of some arity $k\in\NN$ and whose leaves are 
labelled in order by $y_1,\ldots,y_k \in Y$. This product restricts to the 
subcategory of non-graded $\Sigma$-modules, and a symmetric operad $\PP$
is a monoid in this category, whose product we usually denote by 
$\gamma : \PP\circ \PP \longrightarrow \PP$.
We can of course talk about graded or differential graded symmetric operads, 
but symmetric operads will suffice for our purposes. An operad is \new{reduced}
if $\PP(0) = 0$, and we will only consider this kind of operads in what follows.

We recall that there is a unital monad $\mathscr T$ in $\SMod$ and that symmetric
operads are precisely the $\mathscr T$-algebras. It follows in particular that if 
$X$ is any dg $\Sigma$-module, $\mathscr T_X$ is the \new{free operad on~$X$}. 
Concretely, for $n\in \NN$, the space $\mathscr T_X(n)$ consists of rooted trees 
$t$ so that each vertex of $t$ is decorated by an element in $X(d)$, where $d$
is the number of inputs of it. The product $\mu:\mathscr T\circ \mathscr T
\longrightarrow \mathscr T$ is obtained by substitution of trees into vertices, 
and the unit $ \eta : 1\longrightarrow \mathscr T$ sends an element of $X$ to 
the corresponding corolla. We refer the reader to~\cite{LodVal} for details. From
this we can present operads as quotients of free operads by ideals of relations. 
Classical examples include the operads $\textsf{As}$, $\textsf{Lie}$, $\textsf{Com}$, 
$\textsf{PreLie}$, $\textsf{Poiss}$, $\textsf{Ger}$ and $\textsf{Grav}$, whose
presentations can be found in the literature.

\subsection{Algebras over operads}

Recall we have fixed a reduced symmetric operad $\PP$ and write $\PP\hy\Alg$
for the category of dg $\PP$-algebras. 
In particular, every $\PP$-algebra $A$ includes the 
data of a square zero derivation $d$ of $A$, that is, is an endomorphism
$d:A\longrightarrow A$ so that $d^2=0$ and such that for each operation $\mu$ of $O$
of arity $n\in\NN$, $d\mu = \mu d^{[n]}$, where $d^{[n]}$ is the induced 
differential on $A^{\tt n}$. For example,
if $\mu$ is binary (so that $n=2$), the requirement is that
$d\mu = \mu(d\otimes 1+1\otimes d)$. 
Note that Koszul signs will appear in the previous 
formula when evaluating $1\otimes d$ on
elements of~$A$.
Alternatively, we require that the structure
map $\gamma^A : \PP\longrightarrow \End_A$ 
is one of complexes, where we view $\PP$ as a 
dg operad concentrated in degree zero. Adjoint
to this is a map 
$\gamma_A : \PP(A) \longrightarrow A$
which we call the \new{structure map} of $A$.
We will drop the prefix ``dg'' 
and speak simply of $\PP$-algebras. In all of 
what follows, with the exception of the
``tilde'' construction of Proposition~\ref{prop:HHcone}, we will \emph{only}
consider the category of 
$\PP$-algebras $A$ that are positively homologically graded.
In particular our associative algebras are non-unital.

\subsection{Modules over algebras}

Fix a dg $\PP$-algebra $A$ as before. An 
\new{operadic $A$-module} is a dg vector
space $M$ along with a unital action 
$\gamma_M :\PP\circ (A,M)\longrightarrow M$
so that 
$\gamma_M (1\circ(\gamma_A,\gamma_M)) = 
			\gamma_M(\gamma\circ (1,1))$.  
It is useful to note that if $\PP=\mathsf{As}$ 
and if $A$ is a $\PP$-algebra or, what is the 
same, an associative algebra, then an operadic 
$A$-module is the same as an $A$-bimodule and 
\emph{not} a left (or right) $A$-module. 
Similarly, the operadic modules for commutative 
algebras are the symmetric bimodules, and the 
operadic modules for Lie algebras coincide with 
the usual notion of Lie algebra representation.

\subsection{Cofibrations of $\PP$-algebras}
\label{sec: model categories}

For $A,B\in
\PP\hy\Alg$, we write $A\star B$ their
coproduct as $\PP$-algebras. 
The work 
done in~\cite{Hinich1997}, see also~\cite{GetJon,vallette2014homotopy,BerMoer},
shows that non-negatively graded $\PP$-algebras carry a model category structure
whose weak equivalences are the quasi-isomorphisms, and whose fibrations are 
the degree-wise surjections in positive degrees.

{{}} Let us introduce some useful definitions:
\begin{titemize}
\item A $\PP$-algebra $A$ is \emph{quasi-free} if, forgetting its differential, 
it is isomorphic as a graded $\PP$-algebra to the free $\PP$-algebra on a graded vector space $V$. 
That is, if $A \cong \PP(V) = \bigoplus_{n\geq 1} \PP(n)\otimes_{\Sigma_n}V^{\otimes n}$
as a graded $\PP$-algebra.
\item An object $R$ is \new{cofibrant} if the initial map is a cofibration, and \new{fibrant} if 
the final map is a fibration.
All $\PP$-algebras are fibrant (but we 
will not need this) and the cofibrant
objects are the quasi-free triangulated $\PP$-algebras.
\item A cofibration $B\cof A$ is 
\new{elementary of height $d$} if $A$ is 
obtained by adjoining a space of 
generators in degree $d+1$.
\item A morphism $f:B\longrightarrow A$ is
is \new{$n$-connected} if it induces isomorphisms on homology in degree $i<n$ and a surjection in
degree $n$. 
\item An algebra $A$ is \new{$n$-connected} if
the unique map $0\longrightarrow A$ is $n$-connected,
and it is 
\new{$n$-truncated} if $H_s(A)=0$ for 
$s>n$. 
\end{titemize}

We describe next a class of cofibrations
which correspond to the geometric
situation where we attach cells to a space.
Let $A$ be a $\PP$-algebra,
$V$ be a graded vector space concentrated in degree $n\geq 1$,
and $\delta:V\to Z_{n+1}(A)$ be a linear map.
Consider the coproduct of underlying (non-dg) graded algebras
$$B = A \star \PP(V).$$
In $B$, we consider the unique derivation $d$ which restricts to the original differential 
on $A$ and which extends $\delta$ as a derivation.  This derivation squares to zero,
that is, it is a differential.
The inclusion $A \cof  B$ is called an \new{elementary} or \new{standard} cofibration,
see \cite[2.2.3]{Hinich1997}.
A \new{tower of cofibrations} is a $\PP$-algebra $A$ which is the colimit 
of a tower of elementary cofibrations of $\PP$-algebras,
	\begin{equation*}
		A_0 \cof A_1  \cof \cdots \cof A_s \cof A_{s+1} \cof \cdots \cof \varinjlim_s A_s = A.
	\end{equation*}

\begin{definition} 
Let $\PP = \mathcal I \oplus \overline{P}$ be an augmented operad.
The  \new{functor of indecomposables}  $\ind :\PP\hy\Alg\longrightarrow \Cxs $
from $\PP$-algebras to chain complexes is given by taking the cokernel of the
structure map $\gamma_A$ of each $\PP$-algebra $A$ restricted to $\overline{\PP}(A)$. 
We also denote by  $\overline{A}$ this cokernel, 
and call it the \new{space of indecomposables of $A$}. 
\end{definition}

It is clear, for example, that if $A=(\PP(V),d)$ is quasi-free
then $\ind A=(V,d_{(1)})$ where $d_{(1)}$ is the linear part of the differential
$d$ of $A$, and we will use this later. In this way, the (derived) functor 
$\ind A$ captures the ``first order information'' of a cofibrant resolution of
the $\PP$-algebra~$A$.

{{}} 
\begin{definition}
We define the \new{Quillen homology of a 
$\PP$-algebra $A$}, which we write 
$H_*(\PP,A)$, as the left derived functor 
of $\ind$ evaluated at $A$.
 That is, if $B\weq A$  is a weak 
 equivalence and $B$ is a cofibrant $\PP$-algebra, 
the homology of $\ind B$ is by
definition $H_*(\PP,A)$. 
\end{definition}

To illustrate, 
if $A$ is an augmented associative algebra, 
this is just $\Tor^A_{*+1}(\kk,\kk)$.
 Indeed,
we can take $\Omega BA$ as a model of $A$,
and then $\H_*(\PP,A)$ is the homology of
the shifted reduced bar complex $\overline{BA}$
of~$A$, which is precisely
$\Tor^A_{*+1}(\kk,\kk)$.

%%Section 2
\section{A spectral sequence of derivations}\label{sec:2}

This section is the core of the paper.
We pave the way for constructing 
the spectral sequence of Theorem~\ref{thm:SucesionEspectral}, 
the main result.
We start by recalling the relevant facts 
on derivation complexes in 
Subsection~\ref{sec:LongExactSequenceOfDerivations},
and recall the basic facts on tangent cohomology in 
Subsection~\ref{sec:AQ}.
The main result and some corollaries are 
proven in 
Subsection~\ref{ssec:spectral}.
We finish in Subsection~\ref{sec:CollapseResults}
by studying some natural situations in which the 
spectral sequence degenerates.

\subsection{The complex of derivations}\label{sec:LongExactSequenceOfDerivations}

In this section, we collect some general 
facts on complexes of derivations, which
we then use to construct the spectral
sequence of Theorem~\ref{thm:SucesionEspectral}. 
Some applications of the spectral sequence will 
require, for convergence reasons, complexes 
concentrated in non-negative degrees. 
However, the results of this section do not 
need this constraint. 

\begin{definition}
Let $B\longrightarrow A$ be a map
of $\PP$-algebras, and let $M$ be
an operadic $A$-module. We write 
$\Der_B(A,M)$
for the cohomologically	graded complex 
of derivations $A\longrightarrow M$ that
vanish on $B$.
Thus, a homogeneous element $f$ of degree $p$ in $\Der_B(A,M)$ is a linear map from $A$ to $M$ that vanishes on $B$, lowers homological degree by $p$
and is a derivation 
in the sense that for every $\theta\in \PP(n)$ and $a_1,...,a_n\in A$, we have
	\[ f\gamma_A(\theta;a_1,...,a_n) = \sum_{i=1}^n (-1)^{\varepsilon_i} \gamma_M(\theta;a_1,...,f(a_i),...,a_n).\]
We denote such an element in $\Der_B(A,M) $ by $f:A\vert B \longrightarrow M$. The differential of an {homogeneous} $f$
is $\partial f = d_M f -(-1)^{|f|} fd_A$, and as such is of degree $+1$.
\end{definition}

In case $u:A\longrightarrow U$ 
is a map of $\PP$-algebras, which
in particular makes $U$ into an 
$A$-bimodule, we will write $\Der_B(A,U)$ without explicit mention to the map
$u$, which will be understood from
context. In case $U=A$ and $u$ is the
identity, we write this complex simply
by $\Der_B(A)$. Observe that in this case, the differential $\partial F$ is
given by the bracket $[d_A,F]$, and
that $\Der_B(A)$ is a dg Lie algebra
under the Lie bracket of derivations.
Moreover, for each operadic $A$-module
$M$, the complex $\Der_B(A,M)$ is
a left Lie module over $\Der_B(A)$.

The following lemma says that this 
functor is
well behaved when its arguments
are cofibrations of algebras.
The resulting exact sequence is
called the \new{Jacobi--Zariski sequence} 
of the triple $B\longrightarrow A\longrightarrow A'$. It is 
straightforward to see
that it is functorial
on maps of algebras
$A'\longrightarrow U$.  

\begin{lemma}\label{lemma:SequencesOfDerivations} 
	Let $B\longrightarrow A\longrightarrow A'$ be a sequence 
	of morphisms of $\;\PP$-algebras, and
	let $u:A'
	\longrightarrow U$ be a map of $\;\PP$-algebras. Then, there is a 
	left exact sequence 
	\[
	\begin{tikzcd}
	0 \arrow[r] &
	\Der_A(A',U)
	\arrow[r] &
	\Der_B(A',U) 
	\arrow[r] &
	\Der_B(A,	U) 
	\arrow[r,dashed] & 0
	\end{tikzcd} \]
and it is exact if $A\longrightarrow A'$
is a cofibration.	
\end{lemma}

\begin{proof}
	The map $A\longrightarrow A'$ induces a map $\Der_B (A',U) 
	\longrightarrow \Der_B(A,U)$ by precomposition
	whose kernel is tautologically isomorphic to $\Der_A(A',U)$. 
	Hence, it suffices to show
	the last map is surjective in
	case $A\longrightarrow A'$ is a
	cofibration. To do this, we observe that
	we can always lift a derivation $A\longrightarrow U$ 
	that vanishes on $B$ to one in 
	$A'$ by declaring it to vanish on the module of generators
	of $A\longrightarrow A'$, and we will always assume, for 
	consistency,
	that this is our choice of lift. This
	finishes the proof of the lemma. \end{proof}

\subsection{Tangent cohomology}\label{sec:AQ}

Let us now introduce the cohomology theory that will concern us
and serves as the unifying concept to present our results. 
Building on previous work of Getzler--Jones~\cite{GetJon} 
and Goerss--Hopkins~\cite{Goerss2000}, J.~Millès~\cite{Milles} gave the definitive
treatment for tangent cohomology for algebras over algebraic operads.

{{}} For a morphism of $\;\PP$-algebras, 
$B\longrightarrow A$ and an operadic $A$-module
$M$, we define the {tangent cohomology of $A$ 
relative to $B$ with values in $M$} as follows. 
Let $P\longrightarrow Q$ be a cofibrant replacement 
of the map $f:B\longrightarrow A$: by this we
 mean that this map is a cofibration between
 cofibrant $\PP$-algebras that fits into
 a commutative diagram
 $$ 
 \begin{tikzcd}
 P\arrow[r]\arrow[d] & Q \arrow[d]\\
 B\arrow[r] & A
 \end{tikzcd}$$
 where the vertical maps are quasi-isomorphisms. 
 The second vertical map
 makes $M$ into an operadic $Q$-module,
 and hence we can consider the complex
 $ \Der_P(Q,M) $
 of derivations $Q\longrightarrow M$
 that vanish on $P$. 
 
 \begin{definition} The cohomology of
 this complex is, by definition, the
 \new{tangent cohomology of 
 $A$ relative to $B$ with values in $M$},
 and we write it $\AQ^*_B(A,M)$. In 
 case we take $M=A$, we will speak about
 the \new{tangent cohomology of 
 the map $f:B\to A$}, and write it
 $\AQ^*(f)$.
 \end{definition}
 
Note that the Jacobi--Zariski sequence
shows that for a triple $B\longrightarrow A
	\longrightarrow A'$ of algebras, where
$A\longrightarrow A'$ is a cofibration,
we have
	a corresponding long exact sequence
	\[ 
	 \AQ^*_A(A',-) 
	 	\longrightarrow 
	 		\AQ^*_B(A',-) 
	 	\longrightarrow 	
	 	\AQ^*_B(A,-) \longrightarrow 
	 	\AQ^*_A(A',-)[1]. \]
	  
\begin{rmk}\label{2.1.2} Consider the situation when $A$ is an
associative algebra (without
differential), which we think of as concentrated in degree zero, and
take $Q\weq A$ a cofibrant replacement of 
$A$ in $\mathsf{Ass}\hy\Alg$, the cofibration
category of dga
algebras. Then we have identifications for $n\in\mathbb N$,
\[ \AQ^n(A,A) = 
\begin{cases}
\HH^{n+1}(A) & \text{ for $n\geqslant 1$}\\
\Der(A) &\text{ for $n=0$.} \\
\end{cases} \]
where $\HH^*(A)$ 
are the Hochschild cohomology groups of $A$; see~\cite{LodVal}*{Sec. 12.3}. 
\end{rmk}

The way to remedy this difference between
tangent cohomology of associative
algebras and their classical Hochschild
cohomology is as follows.

\begin{lemma} Let $\Ad_Q:
Q\longrightarrow \Der(Q)$ be the adjoint map of $Q$,
where $Q$ is any cofibrant replacement of the associative algebra concentrated in degree zero $A$. 
We have an isomorphism of
(shifted) Lie algebras $H^*(\cone(\Ad_Q)) \simeq 
\HH^*(A)$. 
\end{lemma}

Observe that $\cone(\Ad_Q)$ is a (shifted) Lie
algebra obtained as the semidirect product of $Q$ and 
$s^{-1}\Der(Q)$, where the bracket is given, for 
$x+\varphi$ and $y+\psi \in \cone(\Ad_Q)$, by the
formula $[x+\varphi,y+\psi] = [x,y]+\varphi(y)-
\psi(x)-[\varphi,\psi]$. Notice that in general 
the cone of a morphism of dg Lie algebras does 
not carry a dg Lie algebra structure, but it
can be canonically endowed with an $L_\infty$-algebra
structure lifting that of the domain and the 
usual differential of the mapping cone, as
proved in~\cite{MR2361936}.

\begin{proof}
Since $H_*(Q) = A$, the long exact
sequence of the cone complex gives
us isomorphisms 
$H^n(\cone\Ad_Q) = H^{n-1}(\Der Q)$
for $n\geqslant 2$, giving the claim
of the theorem in that range. The
remainder of the long exact sequence
is a four term exact sequence
\[0\longrightarrow H^0(\cone\Ad_Q) \longrightarrow H^0(Q) 
	\longrightarrow H^0(\Der(Q))
		\longrightarrow H^1(\cone\Ad_Q) \longrightarrow 0.\] 
 Now observe that $H^0(\Ad_Q)$ is 
 just the adjoint map $A\longrightarrow
 \Der(A)$ of $A$. It follows the above exact sequence identifies $H^0(\cone(\Ad_Q))$ with
 the kernel of the map $A\longrightarrow \Der(A)$ and $H^1(\cone \Ad_Q)$ with
 its cokernel, which is what we wanted.  
We used that there is a quasi-isomorphism $\Der(Q)\simeq \Der(Q,A)$.
This is because for any weak equivalence $f:Q\xrightarrow{\simeq }A$, the 
map $\varphi:\Der(Q)\to \Der(Q,A)$ induced by mapping a derivation $d$ 
to the composition $f\circ d$ induces a quasi-isomorphism.
\end{proof}

It will be useful for us to
have an alternative description of this cone.
It is shown in~\cite{FMT} that to compute
$\HH^*(A)$ ---even in the case we allow 
$A$ to be dg--- one may instead consider 
a complex of derivations obtained from the
algebra $Q$ as follows. If $Q=(T(V),d)$, 
consider the algebra 
\[  \widetilde{Q} = 
	(T(V\oplus \varepsilon),\bar d)
\quad
\text{ where } \quad
|\varepsilon|=-1,\quad\bar d\varepsilon = \varepsilon^2,\quad 
\bar dv = dv +[\varepsilon,v]\]
Note that
$\varepsilon$ is a Maurer--Cartan element
and that $F=(T(\varepsilon),\bar d)$ is acyclic ---that is,
$H_*(F) = \kk$--- since for
each $n\in\NN$ we have that
  $\bar d(\varepsilon^{2n-1})=
\varepsilon^{2n}$.  
 
\begin{prop}\label{prop:HHcone}
For any quasi-free associative algebra $Q=(T(V),d)$
the Lie algebra $\Der(\widetilde{Q})$ is
quasi-isomorphic to the Lie algebra given by the cone of the adjoint map $\Ad_Q$. \end{prop}

\begin{proof}

There is a quotient map
$\pi : \widetilde{Q} \longrightarrow Q$
with fibre $F$ 
and a cofibration $F \longrightarrow
 \widetilde{Q}$, which
 gives us a short exact sequence
\[\begin{tikzcd}
%[column sep = large]
0\arrow[r] &\Der_F(\widetilde{Q}) 
\arrow[r]&
\Der(\widetilde{Q})
\arrow[r]&
\Der(F,\widetilde{Q})
\arrow[r]& 0\end{tikzcd}.\]
On the other hand, we have an identification $\Der(F,\widetilde{Q})=\widetilde{Q}$
and a quasi-isomorphism $\Der_F(\widetilde{Q}) \longrightarrow \Der(Q)$ 
coming from the fact that $F$ is quasi-free
and contractible. The takeaway is a commutative
diagram whose rows are exact and whose
columns but the second consist of quasi-isomorphisms: 
$$
\begin{tikzcd}
%[column sep = large]
0\arrow[r] &
	\Der_F(\widetilde{Q}) 
\arrow[r]\arrow[d]&
\Der(\widetilde{Q})\arrow[d,equal]
\arrow[r]&
\Der(F,\widetilde{Q}) \arrow[d,"\mathrm{ev}_\varepsilon"]
\arrow[r]& 0 \\
0\arrow[r] &\Der(Q) 
\arrow[r]\arrow[d,equal]&
\Der(\widetilde{Q})
\arrow[r]\arrow[d]&
\widetilde{Q}
\arrow[r]\arrow[d,"\pi"]& 0\\
0\arrow[r] &\Der(Q) 
\arrow[r]&
\cone(\Ad_Q)
\arrow[r]&
Q
\arrow[r]& 0\end{tikzcd}$$
This shows that the Jacobi--Zariski short exact 
sequence above is quasi-isomorphic to the
short exact sequence for the cone of $Q$,
and hence that $\Der(\widetilde{Q})$ is
quasi-isomorphic to $\cone(\Ad_Q)$.
\end{proof}

\subsection{The spectral sequence}\label{ssec:spectral}

In this section, we present our main technical tool, a spectral
sequence converging under suitable hypotheses 
to the cohomology of the derivations of a
$\PP$-algebra that is the colimit of a tower of cofibrations. 
We will 
study the differentials in this  
spectral sequence, and give 
natural conditions that ensure 
its convergence in
Section~\ref{sec:CollapseResults}.

\begin{theorem}\label{thm:SucesionEspectral} Let $A$ be the colimit of a tower
	of cofibrations of $\;\PP$-algebras 
	\begin{equation*}
		A_0 \cof A_1  \cof \cdots \cof A_s \cof A_{s+1} \cof \cdots \cof \varinjlim_s A_s = A.
	\end{equation*}		
	There is a functorial right half-plane spectral sequence with first page 
	\[ E^{s,t}_1 = H^{s+t} 	(\Der_{A_s}(A_{s+1},-))
	\;			\xRightarrow{\phantom{m}s\phantom{m}} \; 
	H^{s+t}(\Der_{A_0}(A,-)).\]
\end{theorem}

A consequence of this result is
the following corollary which
involves tangent cohomology.
In fact, we will be mostly concerned in
the situation described in the statement
of this corollary:

\begin{corollary}\label{cor:AQss}
	Let $f: A_0\longrightarrow A$ be a cofibration of $\;\PP$-algebras that
	is a cofibrant replacement of a map of $\;\PP$-algebras $B_0\longrightarrow B$. If
	$f$ is the colimit of a tower of
	cofibrations as above,
	there is a functorial
	right half-plane spectral sequence with first page 
	\[ E^{s,t}_1 = H^{s+t} 	(\Der_{A_s}(A_{s+1},-))
	\;			\xRightarrow{\phantom{m}s\phantom{m}} \;	\AQ^{s+t}_{B_0}(B,-).\] \qed
\end{corollary}

We now proceed to prove Theorem~\ref{thm:SucesionEspectral}. 

\begin{proof}
	Let $M$ be an operadic $A$-module,
	and for each $s\geq0$, consider the exact sequence of derivations induced by the triple $A_s\cof A_{s+1} \cof A$, as in Lemma~\ref{lemma:SequencesOfDerivations}:
	\begin{equation*}
		0 \longrightarrow \Der_{A_{s+1}}(A,M) \longrightarrow \Der_{A_s} (A,M) \longrightarrow \Der_{A_s}(A_{s+1},M) \longrightarrow 0.
	\end{equation*} Form the exact couple associated to the long exact sequence in cohomology,  explicitly given for every $s,t\in\NN$ by 
	\begin{equation*}
		D^{s,t}= H^{s+t} (\Der_{A_s}(A,M)), \qquad \textrm{and} \qquad E^{s,t}= H^{s+t} (\Der_{{A_{s}}}(A_{s+1},M)).
	\end{equation*} The maps $i,j,k$ forming the exact couple $(D,E,i,j,k)$ have bidegrees $(-1,1), (0,0)$ and $(1,0)$, respectively and are described after the proof. In the bidegree pair $(s,t)$, we are denoting by $s$ the filtration degree and by $t$ the complementary degree. The differential  $d_1=jk$ produced on $E=E_1$ is of bidegree $(1,0).$ The general procedure of forming iterated derived couples produces the spectral sequence, finding that  the $r$th differential $d_r$ has bidegree $(r,1-r)$. Thus, differentials entering into a fixed module $E_r^{s,t}$ originate at points outside the right-half-plane for all $r>t$, so eventually vanish. 
	
	We now notice that we have  a decreasing filtration $F^s = \Der_{A_s}(A,M)$ on $\Der_{A_0}(A,M)$ given by those derivations $A \longrightarrow M$ that vanish on $A_s$,
	and the exact couple arises from this filtration, see~\cite{McCleary}*{Proposition 2.11}. The exact sequence $0 \longrightarrow 
	F^{s+1} \longrightarrow 
	\Der_{A_0}(A,M) \longrightarrow
	\Der_{A_0}(A_{s+1},M) 
	\longrightarrow 0$ then shows that $\Der_{A_0}(A,M)/ F^{s+1}
	\cong \Der_{A_0}(A_{s+1}, M)$. Since the tower $A_*$ exhausts $A$, 
	it follows that $F^\infty=0$ i.e., that this filtration is Hausdorff, 
	and thus to show that $RF^\infty = 0$, it suffices for us to show that the 
	canonical map
	\[
	\Der_{A_0}(A,M) \longrightarrow
	\lim_s \Der_{A_0}(A_{s+1},M)
	\]
	is surjective. This is equivalent to the condition that a system of derivations 
	on the tower $A_*$ whose restrictions overlap correctly can be promoted to a
	derivation on the colimit $A$, which also holds. To conclude, we notice that 
	since $F$ is Hausdorff the spectral sequence converges weakly in the sense 
	of~\cite{McCleary}*{Definition 3.1}.

	The spectral sequence is functorial, for 
	if $M\longrightarrow N$ is map of $A$-modules, then the morphism induced between the corresponding 
	short exact sequences of derivations induces a morphism
	of the corresponding exact couples, producing the
	desired  morphism of spectral sequences.
\end{proof}

Another corollary of our main result
is a spectral sequence that in many
situations degenerates and exhibits
tangent cohomology as a 
twisted algebra defined in terms of
Quillen homology and ordinary homology
of $\PP$-algebras, in the same spirit
as the usual way of computing Hochschild
cohomology of an associative algebra
$A$ through a twisted complex 
$\hom_\tau(C,A)$ where $C$ is an
$A_\infty$-coalgebra model of $BA$
with associative twisting cochain
$\tau : C \longrightarrow A$.

To state it, let us
take $A$ to be a cofibrant 
$\;\PP$-algebra of the form
$\PP(V)$ that is a cofibrant replacement
of the $\PP$-algebra $B$, and 
write $A$ as the colimit of the 
tower of cofibrations of the form
\[ \kk  \cof
\PP(V_{\leqslant 0}) \cof \PP(V_{\leqslant 1}) 
\cof \cdots \cof \PP(V_{\leqslant s}) \cof \cdots,\]
where for each $s\in\NN$, we write
$A_{s+1} =\PP(V_{\leqslant s})$ for the
subalgebra of $A$ generated by 
elements of degree at most~$s$. 

\begin{corollary}
	\label{cor:paginaE2}
	Assume that $B$ is $0$-connected. 
	There is a convergent functorial right half-plane spectral sequence with second page 
	\[ E^{s,t}_2 = 
	\hom(H_s(\PP,B),H_t(-))
	\;			\xRightarrow{\phantom{m}s\phantom{m}} \;	\AQ^{s+t}(B,-).\] \qed
\end{corollary}

\begin{proof}
	Let us consider the cofibration
	$\kk =A_0 \cof \PP(V)=A$ and compute the homology of the complex of derivations of $A$ relative to $A_0$. 
	For each $s\in\mathbb N_0$ the cofibration 
	$A_s\cof A_{s+1}$
	is obtained by adding generators in degree $s$, and it is elementary. 
	
	{{}} We can easily
	identify the $E_1$-page: a closed derivation $f$ of degree $s+t$ in $\Der_{A_s}(A_{s+1},M)$
	is determined on 
	the generators of $A_{s+1}$ living in
	degree $s$, whose image are in homological degree $t$. 
	From this and the fact $f(dv)= 0$ 
	for a generator $v$ of degree $s$, we see that $f$ must
	have image in the cycles of $M$. 
	Therefore, we have an identification
	\[ E_1^{s,t}=\hom(V_s,H_t(M)).\]
	Since $B$ is $0$-connected, we can
	assume that $V_0=0$. Using this, we now
	check that the 
	differential on the first page is $\hom(d_{(1)},1)$ 
	where $d_{(1)}:V\longrightarrow V$ is the linear part of the
	differential of $\PP(V)$.
	Indeed, let us take a
	cocycle representative $f : A_{s+1}\mid A_s\longrightarrow M$,
	let $F$ be extension by zero to $A_{s+2}$, so that
	$d_1\cls{f} = \cls{[d,F]}$.
	We now observe that
	if $w$ is a generator of $A_{s+1}\cof A_{s+2}$, then
	$F$ vanishes on $w$. 
	The remaining term is $(-1)^{|F|+1} f(dw)$, 
	and now we
	note that any term in $dw$ that is not linear must be
	a product of lower degree terms, which $f$ vanishes on. Since $d_{(1)}$ computes
	the Quillen homology of $B$,  
	the description of the $E_2$-page
	is what we have claimed. 
\end{proof}

We remark that if both $B$ and 
$M$ in Corollary \ref{cor:paginaE2} have zero differential,
the spectral sequence collapses
and yields an isomorphism 
\[ H^*(\hom(H_*(\PP,B),M)) \longrightarrow \AQ^*(B,M) \]
that exhibits tangent
cohomology as the cohomology of
a twisted complex. This follows
immediately from Lemma~\ref{lemma:VanishingLemmas} which
shows that the $E_1$-page is concentrated in one row.
In this sense, the tower of cofibrations
above is rather
crude, and it is perhaps not unreasonable to
consider other more refined towers to expect a
more meaningful spectral sequence.		

\subsection{Collapse results} \label{sec:CollapseResults}

In this section, we collect some natural conditions that guarantee certain vanishing 
patterns on the $E_1$-page of the spectral sequence constructed in Theorem~\ref{thm:SucesionEspectral}.

\smallskip
 Recall that a $\PP$-algebra $U$ is {$b$-truncated} if $H_p(U)=0$ 
for all $p\geq b+1$, and that a  cofibration $B \cof A$ is {elementary of height $s$} if $A= B \star \PP(V)$ with $V=V_{s+1}$. In this case, one has that $d(V)\subseteq B$.

\begin{lemma} 
	\label{lemma:VanishingLemmas}
	Let $B\cof A$ be an elementary cofibration of height $s$, and $u:A\longrightarrow U$ a morphism of $\,\PP$-algebras. The following holds: 
	\begin{tenumerate}
		\item If $U=U_{\geq k}$ for some $k$, then $\Der_B^p(A,U)=0$ for all $p\geq s+2-k.$
		\item If $U$ is $b$-truncated for some $b$, then $H^p(\Der_B(A,U)) = 0$ for all $p \leq s-b.$
	\end{tenumerate}
\end{lemma}

\begin{proof} The fact that $B\cof A$ is an elementary cofibration implies
that any $F \in \Der_B(A,U)$ is
completely determined by its image on $V$, where $A=B\star TV$ and $V=V_{s+1}$. 

To prove the first statement, 
let us take $F \in \Der_B^p(A,U)$. Since $F$ has degree $-p$
the image of $F$ lies in $U_{s+1-p}$. It follows that
$F$ vanishes if $p\geq s+2-k$.
For the second statement, let us
take a cocycle $F\in \Der_B^p(A,U)$ of degree $p\leq s-t$, and show that $F=\partial G$ for some $G\in \Der_B^{p-1}(A,U)$.

Recall that  $F$ is determined by its image on $V$, and that for any generator $v\in V$ we have $F(v)\in U_{s+1-p}$. Given that $F$ is a cocyle, for any such generator $v$ 
we have that
\[ \partial F(v)=dF(v)-(-1)^pFd(v)=0.\] But $d(v)\in B$, because $B\cof A$ is elementary, so $Fd(v)=0$ and thus $F(v)$ is a cycle in $U_{s+1-p}$. 
Since $p\leq s-b$, it turns out that $F(v)\in U_{\geqslant b+1}$ is a boundary. Let $G(v)\in U$ be such that $dG(v)=F(v)$, which can be done since we work over a field. 
This defines $G\in \Der_B^{p-1}(A,U)$ such that $\partial G=F$,
which is what we wanted. Indeed,
$\partial G$ and $F$ are derivations that
coincide on every $v\in V$, and 
hence coincide on all of~$A$. 	
\end{proof}

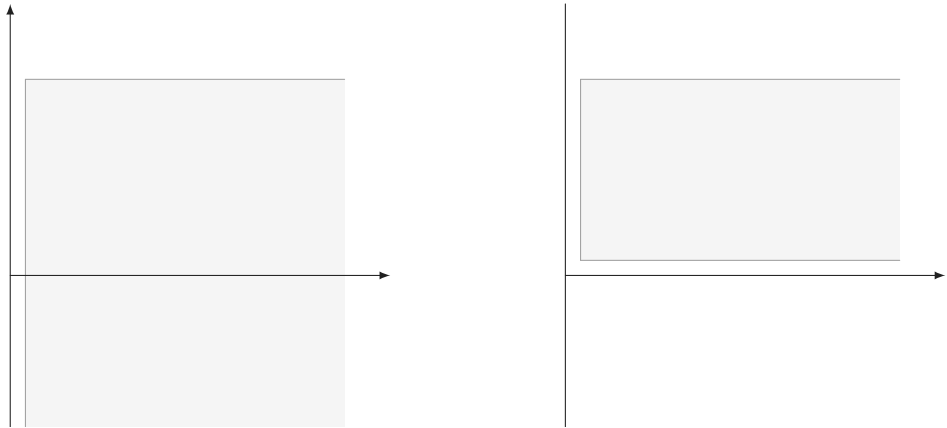
\begin{figure}[h]
	\centering
	\begin{minipage}{0.45\textwidth}
		\centering
		\begin{tikzpicture}[line cap=round,line join=round,>=stealth,x=2.0cm,y=2.0cm, rotate = 270]
		\clip(-2,2.5) rectangle (1,-0.3);
		\fill[color=aqaqaq,fill=aqaqaq,fill opacity=0.1] (1.8,0.1) -- (-1.3,0.1) -- (-1.3,2.2) -- (1.8,2.2);
		\draw [color=aqaqaq] (1.8,0.1)-- (-1.3,0.1);
		\draw [color=aqaqaq] (-1.3,0.1)-- (-1.3,2.2);
		\draw [->,>=latex,color=sqsqsq] (1.8,0) -- (-1.8,0);
		\draw [->,>=latex,color=sqsqsq] (0,0) -- (0,2.5);
		\end{tikzpicture}
	\end{minipage}\hspace{1 em}
	\begin{minipage}{0.45\textwidth}
		\centering
		\begin{tikzpicture}[rotate = 270, line cap=round,line join=round,>=stealth,x=2.0cm,y=2.0cm]
		\clip(-2,2.5) rectangle (1,-0.3);
		\fill[color=aqaqaq,fill=aqaqaq,fill opacity=0.1] (-0.1,0.1) -- (-1.3,0.1) -- (-1.3,2.2) -- (-0.1,2.2);
		\draw [color=aqaqaq] (-0.1,0.1)-- (-1.3,0.1);
		\draw [color=aqaqaq] (-1.3,0.1)-- (-1.3,2.2);
		\draw [color=aqaqaq] (-0.1,0.1)-- (-0.1,2.2);
		\draw [->,>=latex,color=sqsqsq] (-1.8,0) -- (1.8,0);
		\draw [->,>=latex,color=sqsqsq] (0,0) -- (0,2.5);
		\end{tikzpicture}
	\end{minipage}
	\captionsetup{labelformat=empty}
	\caption{A nice filtration and a truncated target.}
	\label{fig:btruncatedcell}
\end{figure}

\begin{definition} A tower of cofibrations 
$A_0 \cof A_1 \cof 
\cdots \cof A_s
\cof\cdots
\cof 
\varinjlim_s A_s = A$ is \new{cellular} if each $A_s \cof A_{s+1}$ is elementary of height $s-1$. That is, if for every $s\in\NN$ the algebra $A_{s+1}$ is obtained
from $A_s$ by freely adjoining a
space of generators in homological degree~$s$.
\end{definition}

{{}} Recall that a quasi-free $\PP$-algebra $\left(\PP(X),d\right)$ is triangulated 
if the space of generators $X$ comes equipped with an exhaustive ascending
filtration
\[
0 = F_0X \subseteq F_1X \subseteq \cdots \subseteq F_pX \subseteq \cdots \cup_p F_pX = X
\]
with the extra property that $d\left(F_pX\right)\subseteq \PP\left(F_{p-1}X\right)$.
Triangulated algebras are good examples of colimits of 
cellular towers of cofibrations, and
in particular the three algebraic models of 
Adams--Hilton, Quillen and Sullivan, 
respectively, are good examples of
colimits of towers of cofibrations relevant
to topology. In Section~\ref{sec:AplicacionesRacionales} we
will exploit this observation to
apply our methods to rational
homotopy
theory.

\begin{corollary}  \label{cor:CollapseResults} 
	Let $A$ be the colimit of a cellular tower of cofibrations of $\;\PP$-algebras
	\[ A_0 \cof A_1 \cof 
\cdots \cof A_s
\cof\cdots
\cof 
\varinjlim_s A_s = A,\] and let $u:A \longrightarrow U$ be a morphism
	of $\;\PP$-algebras. The following holds: 
	\begin{tenumerate}
		\item If $U=U_{\geq k}$ for some $k$, then $E_1^{s,t}=0$ for all $t\geq 1-k$. 
		\item  If $U$ is $b$-truncated for some $b$, then $E_1^{s,t}=0$ for all $t\leq -(b+1)$.
	\end{tenumerate} 
	In particular, if $U$ is $0$-truncated
	and $U=U_{\geqslant 0}$, the spectral
	sequence degenerates at the second page.
\end{corollary}

\begin{proof} The description
of the $E_1$-page of our spectral
sequence and Lemma~\ref{lemma:VanishingLemmas} 
imply the first two statements, if we 
recall that in a cellular tower each
$A_s\cof A_{s+1}$ is elementary of height 
$s-1$. In case $b=k=0$, we obtain that
$E_1^{s,t}=0$ unless $t=0$, so that the
first page is concentrated in one row,
and the claim follows.
\end{proof}

%%Section 3
\section{Applications to rational homotopy theory}\label{sec:AplicacionesRacionales}

In this section, we apply the spectral 
sequence of Theorem~\ref{thm:SucesionEspectral} 
to classical rational homotopy theory using
the models for spaces due to Adams--Hilton and Sullivan. 
We assume the reader has certain familiarity with rational
homotopy theory and recommend the book~\cite{RHT} as
a reference to this subject.	
Since these models are constructed as colimits
of towers of cofibrations of associative and associative
commutative dg algebras, respectively, this context 
is a natural starting point to consider what the
specialization of our spectral
sequence look like. 

We remark that, unlike what we do here,
one can use this spectral sequence over positive
characteristic provided the corresponding operad is
well-behaved in that setting, such as the associative one. 
This may provide with further applications to $p$-local 
homotopy theory by exploiting, for instance, 
the Adams-Hilton construction carried over the ring of
integers localized at a prime;
see for example~\cite{AnickBook, AnickSpheres, NeisendorferProduct}.

\bigskip

\noindent\textit{\textcolor{newcol}{Conventions for the 
section.}} In this section, we change the bidegree in the 
spectral sequence of Theorem
\ref{thm:SucesionEspectral} by $(s,t)\mapsto (s,-t)$. 
This gives more natural formulas, avoiding 
negative signs in the applications that follow. 
The spectral sequence is still right-half-plane,
but the differential $d_r$ changes its bidegree to 
$(r,r-1)$, and under convergence  assumptions, the  $E_\infty$-page 
recovers the cohomology of the target 
by the formula 
\[ (\operatorname{gr} H)^p = \bigoplus_{s=p+t} E_\infty^{s,t}.\] 
We also revert to using the qualifier
\emph{dg} in front of algebras to avoid
any confusion when dealing with
(co)chain algebras, for example.
\subsection{Three flavours of homology}

{{}} Let us now draw up 
some simple 
connections between the existing 
geometrical models of spaces from
the work of F. Adams and P. Hilton,
D. Quillen, and D. Sullivan,
 and 
the various homology functors we
have at hand, namely:

\begin{tenumerate}
\item The functor $H_*(-):\PP\hy\Alg
\longrightarrow \PP\hy\Alg$
 that assigns to a
(dg) $\PP$-algebra $A$ its homology
 groups $H_*(A)$. As our notation
suggests, these are also (non dg)
$\PP$-algebras.

\item The functor $H_*(\PP,-):\PP\hy\Alg
\longrightarrow \PP_\infty\hy\Cog$ that assigns to a
$\PP$-algebra $A$ its Quillen homology
groups $H_*(\PP,A)$. As our notation
suggests, these are homotopy
$\PP$-coalgebras. We will use this
only for $\PP$ the associative operad.
\item The assignment  (but not functor)
$\AQ_* : \PP\hy\Alg \longrightarrow 
\mathsf{Lie}\hy\Alg$ that assigns to a
$\PP$-algebra $A$ its tangent
 cohomology groups with values in
itself $\AQ_*(A,A)$. 
\end{tenumerate}

{{}} On the other hand,
we have three classical constructions
on topological spaces:

\begin{tenumerate}
	\item The \new{Quillen functor}
	$\lambda : \mathsf{Top}_{*,1} \longrightarrow
	\mathsf{Lie}\hy\Alg$ that assigns to a pointed
	$1$-connected $X$ a dg
	Lie algebra $\lambda(X)$
	that models the Lie algebra of
	homotopy groups $\pi_*(\Omega X)$. 
	\item The \new{Sullivan functor}
	$\Sull :  \mathsf{Top}  \longrightarrow
	\mathsf{Com}\hy\Alg$ that
	assigns to a space $X$ a 
	commutative dga 
	algebra $\Sull^*(X)$ that is a model
	of the cochain algebra $\Sing^*(X)$. 
	\item The \new{Adams--Hilton construction}
	$\mathcal{A} :  \mathsf{Top}_{*,1} \longrightarrow
	\mathsf{Ass}\hy\Alg$ that assigns to a pointed
	$1$-connected $X$ a dga
	algebra $\mathcal{A}_*(X)$ that is a model of the
	Pontryagin
	algebra of chains $\Sing_*(\Omega X)$.
\end{tenumerate}

%Mention also the existence of a second Lie algebra functor
%$\mathcal L: \mathsf{Top} \longrightarrow
%\mathsf{cLie}\hy\Alg$ defined on arbitrary spaces that
%in a precise sense (due to the absence of a base point here), agrees with Quillen's in the simply-connected case \cite{BFMT}. It 
%lands in the category of complete Lie algebras.

In the following table we record some of the
relations between these algebraic models of
spaces and the three flavours of homology 
above. 

\bigskip

\begin{table}[h]
\centering
\begin{tabular}{@{} l *3c @{}}
\toprule 
& \multicolumn{3}{c}{Homology theory} \\
 \cmidrule{2-4}
 \multicolumn{1}{c}{Models}    & Ordinary   & Quillen   & Tangent   \\ 
\midrule
 Quillen  & $\pi_*(\Omega X)$ & $H_{*+1}(X)$ &  $\pi_*(LX)$ \\ 
 Sullivan--de Rham & $H^*(X)$ & $\pi_{*}(X)^\vee$ & $\pi_{*}(\FF(X,X))$ \\
 Adams--Hilton & $H_*(\Omega X)$ & $H_{*+1}(X) $ & $H_*(LX)$\\
 \bottomrule
 \end{tabular}
\end{table}

We point the reader to~\cite{BlockLazarev,
SchStash,BerglundSaleh} for useful references
where these relations are studied in detail. 
We now proceed to recall the last three
functors above and apply our operadic
formalism to obtain spectral sequences
to compute their tangent 
cohomology groups, which we will
identify with invariants of known
geometrical objects.

\subsection{The Adams-Hilton model}\label{sec:AdamsHilton} 

{{}} In~\cite{FMT} the authors exploit the
identification of Hochschild cohomology and the
homology of derivations of Remark~\ref{2.1.2} along with 
the identification of the loop homology of a closed
orientable manifold with the Hochschild cohomology
of $\Sing_*(\Omega X)$ to compute the loop bracket of~$X$. This is done under the hypothesis the 
identification of Cohen and Jones~\cite{CohenJones} 
between loop homology and Hochschild
cohomology of $\Sing^*(X)$ commutes with the bracket.
This has been confirmed in~\cite{MR2415345}, where the
stronger statement that the induced map is a morphism
of BV-algebras was proved.

{{}} Explicitly, one can compute Hochschild cohomology of $\Sing_*(\Omega X)$ 
through the Adams--Hilton model $\mathcal A_*(X)$ of $X$,
and then the loop space bracket as the Lie bracket
of derivations of its associated Lie algebra of
derivations $\Der(\mathcal A_*(X))$. 
Since the model $\mathcal A_*(X)$ has generators 
that are in bijection with the cells of a CW 
decomposition of the manifold $X$, this method lends itself
quite nicely to computations.
We now recall the 
construction of Adams and Hilton.

{{}} Let $X$ be a CW complex with exactly one $0$-cell,
no $1$-cells, and such that all the attaching maps of
higher dimensional cells are based with respect to 
this only $0$-cell. In~\cite{AH}, the authors construct
a cofibrant model $\mathcal A_*(X)$ of the dg algebra $\Sing_*(\Omega X)$,
where $\Omega X$ is the Moore loop-space of~$X$. In the following, for each
$n\in\NN$ write $X_n$ for the $n$-skeleton of $X$.

\begin{theorem*}[Adams--Hilton~\cite{AH}]
There is a cofibrant model  \[ f_X : \mathcal A_*(X)=(TV,d) \longrightarrow (\Sing_*(\Omega X),d)\] of the Pontryagin dga
algebra of $X$ such that for each non-negative integer $n$,
\begin{enumerate}[leftmargin = 1.5 cm, label = \emph{(L{\arabic*})}]
\setlength{\itemsep}{0pt}
  \setlength{\parskip}{0pt}
\item the space $V_n$ has basis the
$(n+1)$-cells of $X$, so that $V= V_{\geqslant 1}$,
\item the map $f_X$ restricts
to quasi-isomorphisms $\mathcal{A}_*(X_n)
\longrightarrow(\Sing_*(\Omega X_n), d)$,\item if $g : S^n\longrightarrow X_n$ is the attaching map of
a cell $e$ in $X$, then
\[ \left(f_{X}\right)_*\cls{dv_e} = K\cls {g},\] 
where $K$ is the isomorphism  
$\pi_n X_n \to \pi_{n-1}\Omega X_n$ followed
by the Hurewicz map. \hfill \qed
\end{enumerate} 
\end{theorem*}
{{}} 
These conditions determine the dg algebra $\mathcal A_*(X)$
uniquely up to isomorphism, and we call it \new{the Adams--Hilton
model of $X$}. Of course, this model depends on the
CW structure of $X$, which we take as part of the input
data for its construction. As we
noted,
homology of the shift of indecomposables $s\ind \mathcal A_*(X)$ of the
Adams--Hilton model is the (reduced)  homology $H_*(X)$. In other words, $s\ind \,\mathcal A_*(X)$ is the reduced cellular chain
complex $\Cell_*(X)$ of $X$.

{{}} We remark that in~\cite{AH} the authors actually
produce a model of the fibration sequence of a CW complex $X$. That is, there is a commutative 
diagram of complexes of $\kk$-modules
$$
\begin{tikzcd}[row sep = large]
\mathcal A_*(X)  \arrow[r]\arrow[d] &
		 \Cell_*(X) \otimes \mathcal A_*(X) \arrow[r]\arrow[d] &
		 		\Cell_*(X) \arrow[d] \\
\Sing_*(\Omega X)  \arrow[r] &
			\Sing_*(LX) \arrow[r] &
					\Sing_*(X)  
\end{tikzcd}
$$
where all the vertical maps are quasi-isomorphisms,
the top row is the classical ``algebraic fibration'' 
coming from a cobar construction,
and the bottom row is obtained
by applying the singular
chains functor to the based
path-space
fibration. Moreover, the first
vertical map is a map of dga algebras
and 
the second vertical map is a map
of right modules with respect to
the obvious action of $\mathcal{A}_*(X)$
on $\Cell_*(X)\otimes \mathcal A_*(X)$
and the action of $\Omega X$ on $LX$. 
The Adams-Hilton model 
and, more generally, non-commutative
dg algebra models have proven successful
in the study of the $p$-local homotopy theory of spaces, for example; see the
article~\cite{AnickSpheres} and the
book~\cite{AnickBook}.

{{}} We now show how to
run our spectral sequence to
compute loop space homology of a space~$X$ by using
the filtration by degree in the Adams--Hilton model. In doing so, we obtain the 
following result, which the reader can 
compare with the spectral sequence of 
S. Shamir~\cite{Shamir} and to
the eponymous spectral sequence of 
J.-P. Serre. We remark 
that our spectral sequence has 
differentials that we can control provided we can
control the Adams--Hilton model. This gives us
a spectral sequence of the same shape as that
in~\cite{Cohenloops}, produced by R. L. Cohen, 
J. D. S. Jones and J. Yan. 

\begin{theorem}\label{thm:Adams-Hilton}
Let $X$ be a CW complex with exactly one $0$-cell, no $1$-cells
and all whose attaching maps are based with respect to the only
$0$-cell. There is a convergent 
first quadrant spectral sequence with
\[ E_2^{s,t}  = \hom(H_s(X), H_t(\Omega X)) 
\;			\xRightarrow{\phantom{m}s\phantom{m}} \;
					H_{s+t}(LX),
		 		\]
\end{theorem}

\begin{proof}
Since $X$ is simply connected, the
Adams--Hilton model of $X$ is a
$0$-connected dga algebra, and then
Corollary~\ref{cor:paginaE2} applies.
We already noted that 
the Quillen homology of $\mathcal{A}_*(X)$
is $H_{*+1}(X)$, while
its homology is $H_*(\Omega X)$.

To conclude, we need only address 
the fact we have replaced the
positive homology groups of $X$ 
with the homology groups of $X$,
and the
identification of the
target of our spectral sequence.
To do the first, we recall from 
Proposition~\ref{prop:HHcone} that we may 
pass from tangent cohomology
to Hochschild cohomology
of the dga algebra $\mathcal A_*(X)$
by taking the cone of the adjoint
map 
\[ \Ad : \mathcal A_*(X)\longrightarrow 
	\Der\mathcal A_*(X). \]
We then filter the cone by only filtering the summand corresponding to
derivations. This has little effect
in our spectral sequence, except that it
adds the summand corresponding to
the homology of $\mathcal A_*(X)$,
namely 
\[ H_*(\Omega X) = \hom(H_0(X),H_*(\Omega X)), \]
in the second page and throughout
the computation, and produces the
desired shift (there is a shift in the cone).

Having addressed
this, we now recall from~\cite{LodayLoops}*{Chapter 4} 
and \cite{CohenJones} that the Hochschild
cohomology groups of the Pontryagin
algebra $\Sing_*(\Omega X)$ of $X$
are functorially isomorphic to
the homology groups of the free
loop space $LX = \operatorname{Map}(S^1,X)$ of~$X$. 
\end{proof}
{{}} 

\begin{rmk}
{{}} Observe that if $\AA_*(X)$ 
is minimal or, what is the same, if the linear 
part of its differential is 
zero, then the description of the $E_2$-page is greatly 
simplified. If there are only cells in even degree or 
only in odd degree, then $\AA_*(X)$ 
will be minimal, for 
example. 
\end{rmk}

\subsection{Multiplicative structure}

For $X$  a simply connected
closed oriented manifold of dimension $m$, write
$\mathbb H_*(LX) = H_{*+m}(LX)$. 
 We recall
from~\cite{String} that there is a
\emph{loop product}
$ -\bullet- :\mathbb H_*(LX)\otimes
	\mathbb H_*(LX) \longrightarrow \mathbb H_*(LX)$. 
	One  of the main results of that
	paper is as follows:
	\begin{theorem*} For every simply  connected closed oriented
	manifold $X$, there is a natural isomorphism of graded  
	commutative associative
	algebras
	\[ (\mathbb H_*(LX),-\bullet -)
		\longrightarrow (\HH^*(\Sing^*(X)),-\smile -). \]
	\end{theorem*}
	
Following the notation of Section~\ref{sec:AdamsHilton}, let
$\AA_*(X) = (TV,d)$ be an Adams--Hilton model for $X$, so that the cone
of the adjoint map of $\AA_*(X)$ computes $\mathbb H_*(LX)$. In our spectral
sequence, the $E_2$-page has the form
\[ E_2^{s,t} = \hom(H_s(X),H_t(\Omega X)) \]
and, since $H_*(X)$ is a coassociative coalgebra and $H_*(\Omega X)$
an associative algebra, this page inherits an associative 
\new{convolution product} given by the following composite:
\[ H_*(X) \stackrel{\Delta}\longrightarrow H_*(X)\otimes H_*(X) 
	 \stackrel{f\otimes g}\longrightarrow H_*(\Omega X)\otimes H_*(\Omega X)
	  	\stackrel{\mu}\longrightarrow H_*(\Omega X).\]
Observe that in case we consider elements with domain in $H_0(X)=\mathbb Q$,
this identifies with the Pontryagin product of $H_*(\Omega X)$.	
We remark that this is in line with Theorem 1 in~\cite{Cohenloops},
and also point the reader to~\cite{Meier2011}.

\begin{theorem}\label{thm:loopproduct}
The spectral sequence of Theorem~\ref{thm:Adams-Hilton} is multiplicative.
The convolution product on the $E_2$-page converges to the cup product
on the Hochschild cohomology of $\AA_*(X)$, which equals the Chas--Sullivan
product in the loop homology groups of $X$.  
\end{theorem}

\begin{proof}	
The cup product on $\HH^*(\AA_*(X))$ is induced
from the differential $d$ of $\Der(\AA_*(X))$ as follows.
Let $f$ and $g$ be linear maps $V\longrightarrow TV$
which correspond to derivations $F$ and $G$. Then
$F\smile G$ is determined by the derivation induced
from the brace operation $\{d;f,g\}$~\cite{mescher2016primer,HinichTamarkin}.
Explicitly, this is obtained
by all possible ways of inserting $f$ and $g$ (in this order)
into the operators $d=d_2+d_3+d_4+\cdots$. The
first few terms are as follows:
\[ \{d;f,g\} = d_2(f,g) + d_3(f,g,1) + d_3(f,1,g) + d_3(1,f,g)+ \cdots.\]
Using the filtration of Theorem~\ref{thm:Adams-Hilton}
on the $E_2$-page
we are left only with the term induced in homology by the
derivation associated to $d_2(f,g)$, which on $V$ restricts
precisely to $\mu\circ (f\otimes g)\circ\Delta$, which is
what we wanted.
\end{proof}

We point out a similar description
for a Lie bracket in the cohomology
groups of derivations of Sullivan
models of spaces through
brace operations is given
in~\cite{BuijsMurillo}*{Theorem 3}.

\subsection{The Sullivan model} \label{sec:Sullivan}

{{}} The Sullivan model of a topological space has proved to be 
quite successful and versatile in studying rational 
homotopy theory and its connection to other fields.
We refer the reader to \cite{RHT,AlgebraicModelsGeom}
for a survey on the various results obtained through this formalism.
In this section, we study the spectral sequence 
of Theorem~\ref{thm:SucesionEspectral} in this context, that
is, when considering cofibrant
commutative dga algebras that model the rational
homotopy type of topological spaces
and fibrations between these.

\medskip

\textit{\textcolor{newcol}{The cohomological conventions.}}
As a chain algebra, a Sullivan algebra 
is concentrated in non-positive degrees, and the
complexes of derivations
of~\cite{Fel10}, which we will be using,
are  the same as ours.  Since we intend to apply
Corollary~\ref{cor:CollapseResults} to Sullivan algebras, which are
naturally cohomologically graded and concentrated 
in non-negative degrees, we substitute 
the ``truncated'' condition in Section \ref{sec: model categories}
by the following. 

\begin{definition}
Let $b\in\NN$. A $\PP$-algebra $A$ is 
\new{$b$-truncated} if $H^p(A)$
vanishes for all $p\geq b+1$. 
\end{definition}

For example, an $n$-manifold has an 
$n$-truncated Sullivan model. Using these cohomological conventions,
the items of Lemma~\ref{lemma:VanishingLemmas} read as
follows:

\begin{lemma} Let $U$ be a cohomologically graded $\PP$-algebra.
\begin{tenumerate}
	\item If $U=U^{\geq k}$ for 
	some $k$, then $\Der_B^p(A,U) = 0 $ 
	for all $p \leq k-s-2,$
	\item If $U$ is $b$-truncated for some $b$, then 
	$H^p(\Der_B(A,U)) = 0$ for all
	$p\geq b-s$.
\end{tenumerate}
\end{lemma}

Similarly, the items of
	Corollary~\ref{cor:CollapseResults}  read as follows,
	where again $t$ is replaced by $-t$:

\begin{lemma}Let $U$ be a cohomologically graded $\PP$-algebra.
	\begin{tenumerate}
		\item If $U=U^{\geq k}$ for some $k$, 
		then $E_1^{s,t}=0$ for all $2s\leq k+t-1$.
		\item If $U$ is $b$-truncated for some $b$, 
		then $E_1^{s,t}=0$ for all $2s\geq t+b+1.$
	\end{tenumerate}
\end{lemma}

Since Sullivan algebras are concentrated 
	in non-negative 
	degrees, we will have that $E_1^{s,t}=0$ whenever
	$t\geq 2s+1$. 
	If moreover $X$ is 
	an $n$-manifold, or any space whose rational cohomology 
	is concentrated in degrees $\leq n$, then we 
	have the sharper vanishing of $E_1^{s,t}$ 
	whenever $t+n+1 \leq 2s$.

\bigskip
{{}} Let us fix a fibration $p:E\longrightarrow B$, where $E,B$ are 
1-connected CW-complexes and $E$ is finite, and let us take
 $(\Lambda W,d) \cof (\Lambda W\otimes \Lambda V,D)$
 a
relative Sullivan model of this fibration. In other words, 
the following diagram of cdga algebras commutes and the vertical maps
are quasi-isomorphisms of cdga algebras:
\begin{center}
	\begin{tikzcd}[row sep = large]
	\Sull(B) \arrow[r, " \Sull(p)"]                                 & \Sull(E)                                                       \\
	{(\Lambda W,d)  } \arrow[r, hook] \arrow[u, "\simeq"] & {(\Lambda W\otimes \Lambda V,D)} \arrow[u, "\simeq"].
	\end{tikzcd}
\end{center} 
The main result in~\cite{Fel10} 
shows that this cofibration codifies
the homotopy type of $\Aut_1(p)$, the  
component of the identity in the topological monoid $\Aut(p)$ of fibre-homotopy self equivalences $f: E\longrightarrow E$. We recall this means
that $f$ is a homotopy equivalence
such that $pf=p$. 

\begin{theorem*}[Theorem 1 in~\cite{Fel10}]
There is an isomorphism of graded Lie algebras
\begin{equation*}
H^*(\Der_{\Lambda W}(\Lambda W\otimes \Lambda V))
	 \longrightarrow 
	 	\pi_*(\Aut_1(p)). 
\end{equation*} 
\end{theorem*}

In the language of this paper, this result
states the following:

\begin{theorem*}
The tangent cohomology groups of the map
\[
	 \Sull(p):\Sull(B) \longrightarrow \Sull(E)
		\] are
isomorphic, as a graded Lie algebra,
to the rational Samelson Lie algebra $\pi_*(\Aut_1(p))$: there is an
isomorphism of graded Lie algebras
\[
	 \AQ^*(\Sull(p))  \longrightarrow \pi_*(\Aut_1(p)).
	\]
\end{theorem*}

It is useful to remark that, since each connected component of $\Aut(p)$ has the same rational 
homotopy type, the construction above 
determines the rational homotopy type
of the space $\Aut(p)$ in terms of the
rational homotopy type of $\Aut_1(p)$
and the group structure in $\pi_0(\Aut(p))$. See~\cite{Fel10} 
for complete details.

We now observe that 
we can exhibit the cofibration
$(\Lambda W,d) 
	 \longrightarrow 
	 (\Lambda W\otimes \Lambda V,D)$
	 as the tower of
cofibrations obtained by adding the
cells of $\Lambda V$ ``one by one''.
With this in mind, let us put
the technical tools developed in
Section~\ref{sec:LongExactSequenceOfDerivations}
into the appropriate context to 
apply them to Sullivan algebras.

The construction of the relative Sullivan model~\cite[Proposition 15.6]{RHT}, 
does not require any finite type hypotheses on 
the spaces involved in the fibration, 
see \cite[Theorem 3.1]{RHT2}.
On the other hand, for any simply connected 
space $X$ with Sullivan model $(\Lambda V,d)$,
we do need $\pi_*(X)$
finite dimensional in each degree in order to
have an identification $V^* = \pi_*(X)$.
Otherwise, we can only conclude that 
there is an isomorphism
$V^* \longrightarrow \operatorname{hom} (\pi_*(X), \Q)$.  
Since in most applications we will 
take $X$ to be a finite type CW-complex,
a classical result of Serre, see
for instance \cite[Theorem 20.6.3]{Tammo}, let us identify
$V^*=\pi_*(X)$,
and we will write it like this in the statements.

\begin{theorem} \label{thm:SullivanFibrations}
	Let $F\hookrightarrow E \xrightarrow{p} B$ be 
	a fibration of $1$-connected CW-complexes, 
	with $E$ finite. There is a convergent
	spectral sequence with 
	\begin{equation*}
	E_2^{s,t} =  \hom(\pi_s(F),H^{t}(E)) \; \xRightarrow{\phantom{m}s\phantom{m}}
	 \; \pi_{s-t}(\Aut_1(p)).
	\end{equation*}
\end{theorem}

\begin{proof} 
Let $(\Lambda W,d) \cof (\Lambda W\otimes \Lambda V,D)$ be a relative Sullivan model of the fibration. The 
	Sullivan condition provides a filtration $V(k)$ 
	of $V$ such that  
	\begin{equation}\label{ecu:SullivanCon}
	DV(k)\subseteq  \Lambda W\otimes \Lambda V(k-1) \quad \textrm{ for all } k\geq 0.
	\end{equation} Since $W=W^{\geq 2}$, we can 
	assume that its Sullivan filtration is given by 
	degree, $W(k)=W^{\leq k}$. Equally, we can assume the same degree filtration on $V$. 
	It thus makes sense to consider the cellular 
	tower of cofibrations of cdga's given by 
	$A_s=\Lambda W\otimes \Lambda V^{<s}$, whose 
	colimit is the relative Sullivan algebra 
	\[(\Lambda W,d) 
		\cof (\Lambda V\otimes \Lambda W,D). \] 
	
We are in the cohomological
situation of Corollary~\ref{cor:paginaE2},
and this gives us the second page if
we note that the cohomology groups of
$\Lambda W\otimes \Lambda V$ are those of
$E$. On the other hand, the quotient algebra
$(\Lambda V,d')$ is a model for the fibre, 
so its Quillen homology  gives the homotopy
groups of~$F$. To conclude, recall 
that the target of the spectral sequence are the cohomology 
groups of $\Der_{\Lambda W}(\Lambda W\otimes \Lambda V)$, 
so an application of~\cite[Theorem 1]{Fel10} finishes the proof.
\end{proof}

Applying this result to the 
trivial fibration $X\longrightarrow *$ yields the following result.

\begin{corollary}\label{cor:fibraciontrivial} 
	Let $X$ be a finite, $1$-connected CW-complex. There 
	is a convergent first quadrant spectral 
	sequence with $
	E_2^{s,t} =  \hom(\pi_s(X) ,H^{t}(X)) \; \xRightarrow{\phantom{m}s\phantom{m}}  \; \pi_{s-t}(\Aut_1(X))$.
\end{corollary}

To obtain the corollary above, we just 
fed a specific fibration to the spectral
sequence of Theorem~\ref{thm:SullivanFibrations}.
The following table collects some other 
fibrations and the target
of the corresponding spectral sequence.

\medskip
\begin{table}[h]
\centering
\begin{tabular}{@{} l*3c @{}}
\toprule 
Input & Shape   &Target     \\ 
\midrule
loop-space fibration & $PX\longrightarrow X$  & $\pi_*(\Omega X)$  \\
trivial fibration &  $F\times X\longrightarrow X$  & $\pi_*(\FF(X,\Aut(F))$  \\
principal $G$-bundle & $E\longrightarrow B$   & $\pi_*(G_\circ)$  \\
 \bottomrule
 \end{tabular}
\end{table}

We remark that the spectral sequence 
of Theorem~\ref{thm:SullivanFibrations}
is multiplicative for a graded Lie bracket 
that, on the second page, identifies 
with the convolution bracket obtained 
from the Lie coalgebra~$\pi_*(F)$ with
the Whitehead cobracket and the commutative
algebra $H^*(E)$ with the cup product. 
A result analogous to Theorem 3 
in~\cite{BuijsMurillo} for 
$\mathcal{F}(X,Y)$ replaced with $\Aut_1(p)$
should then give the
 conclusion analogous to that of 
Theorem~\ref{thm:loopproduct} that
this product converges to the Whitehead
product in the target homotopy groups
$\pi_*(\Aut_1(p))$.

\pagestyle{references}
%\begin{multicols}{2}
\bibliographystyle{plain} 
\bibliography{spectral-sequence-v2-may}

\vfill

\begin{small}
	\myauthor{José M. Moreno-Fernández}{josemoreno@uma.es}
	{
		\textsc{Dep. de Álgebra, } \\
		\textsc{Geometría y Topología},	\\ 
		Facultad de Ciencias, \\
   Universidad de Málaga, \\
   29080 Málaga,  Spain
   							}
\end{small} 
\medskip

\begin{small}
	\myauthor{Pedro Tamaroff}{tamarofp@hu-berlin.de}
	{
		\textsc{Institut für Mathematik},	\\ 
		Johann von Neumann-Haus, 			\\
		Humboldt-Universität zu Berlin, 	\\
		Rudower Chaussee 25, 				\\
		12489 Berlin, Deutschland
									}
\end{small}  
\vspace*{\fill} 

%\end{multicols}
\end{document}